\documentclass[12pt,a4paper]{amsart}

\makeatletter
\@namedef{subjclassname@2020}{%
  \textup{2020} Mathematics Subject Classification}
\makeatother
\newtheorem{theorem}{Theorem}[section]

\theoremstyle{definition}

\newtheorem{example}[theorem]{Example}

\theoremstyle{remark}
\newtheorem{remark}[theorem]{Remark}

\numberwithin{equation}{section}

\begin{document}

\title{Generalized  cross curvature flow}

\author{Shahroud Azami}
\address{Department of  Pure Mathematics, Faculty of Science,
Imam Khomeini International University,
Qazvin, Iran. \\
              Tel.: +98-28-33901321\\
              Fax: +98-28-33780083\\
             }

\email{azami@sci.ikiu.ac.ir}
%
%




\begin{abstract}
In this paper, for a given  compact  3-manifold with an  initial Riemannian metric and  a symmetric tensor, we establish the short-time existence and uniqueness theorem for extension of  cross curvature flow. We give an example of this flow on manifolds.
\end{abstract}

\maketitle

\section{Introduction}
The cross curvature  flow (XCF)  on 3-manifolds  were first introduced by B. Chow and R. S.  Hamilton \cite{CH} as
\begin{equation*}
\frac{\partial g}{\partial t}=-2\epsilon h,\,\,\,\,\,\,g(0)=g_{0},
\end{equation*}
where $h$ is cross curvature tensor, $\epsilon=\pm 1$ is the sectional curvature sign of the metric $g_{0}$. They obtained several monotonicity formula  to show that for an initial metric  with negative sectional curvature the XCF exists for short-time and converge to a hyperbolic metric after an appropriate normalizative  of it.  Then, J. A. Buckland \cite{B},  studied its short-time existence on closed 3-manifolds. Moreover, et al. Cao \cite{C1, C2}, investigated the asymptotic behavior  of the cross curvature flow on locally homogeneous three-manifolds.  Several examples of solutions to  XCF have been obtained by L. Ma and   D. Chen
\cite{MC} for warped product metrics on 2-torus and 2-sphere bundle over the circle. Also, in \cite{a,b} have been obtained some results about XCF.\\The Ricci flow \cite{H} is the most successful example for deforming  a Riemannian structure via a geometric evolution equation  and it is defined as
\begin{equation*}
\frac{\partial g}{\partial t}=-2Ric,\,\,\,\,\,\,g(0)=g_{0},
\end{equation*}
where $Ric$ denotes the Ricci curvature and this flow is a natural analogue of the heat equation for metrics. The existence solution of the Ricci flow  on closed Riemannian manifolds was studied by Hamilton \cite{H} and DeTurck \cite{D}. \\
Another  geometric flow is the Ricci-Bourguignon flow which is a generalization of the Ricci flow and it is defined as follows
\begin{equation}
\frac{\partial }{\partial t}g=-2Ric+2\rho R g=-2(Ric-\rho R g),\,\,\,\,\,\,\,\,\,\,\,\,\,\,\,g(0)=g_{0},
\end{equation}
where $R$ is the scalar curvature  of $g$ and $\rho$ is  a real  constant. The  Ricci-Bourguignon flow was introduced by  Bourguignon \cite{JPB} for the first time in $ 1981$. Short-time
existence and uniqueness for the
solution of the Ricci-Bourguignon flow for $\rho<\frac{1}{2(n-1)}$ on $[0,T)$ have been shown  by Catino  et al. in  ~\cite{GC}.
When  $\rho=0$,
the Ricci-Bourguignon flow is the Ricci flow.\\

Motivated by the above works, in this article we consider a 3-dimensional compact Riemannian manifold $M$ whose metric $g=g(t)$ is  evolving according to the flow equation
\begin{equation}\label{hcc}
\frac{\partial g}{\partial t}=-2\epsilon h+2\rho Rg,\,\,\,\,\,\,g(0)=g_{0},
\end{equation}
where $h$ is the cross curvature tensor  and  $\epsilon=\pm 1$ is the sectional curvature sign of the metric $g_{0}$. We establish the short-time existence of the solution  to (\ref{hcc}) by using  more elementary method than  \cite{K1} on 3-dimensional compact Riemannian manifold $M$ and give some evolution equations of curvatures. This flow called hyperbolic cross curvature flow and  for simplicity, we will  denote it by ECCF in short.
\section{The cross curvature tensor}
Let $(M,g)$ be a 3-dimensional Riemannian manifold. The Einstein tensor is $E_{ij}=R_{ij}-\frac{1}{2}Rg_{ij}$. Raising the indices, define $P^{ij}=R^{ij}-\frac{1}{2}Rg^{ij}$. Let $V_{ij}$ be the inverse of $P^{ij}$. The cross curvature tensor is given by
\begin{equation*}
h_{ij}=\left(\frac{\det P^{kl}}{\det g^{kl}} \right)V_{ij}.
\end{equation*}
Let  denote the volume form  by $\mu_{ijk}$ and rasie indices by $\mu^{ijk}=g^{ip}g^{jq}g^{kr}\mu_{pqr}$ and normalize such that $\mu_{123}=\mu^{123}=1$. Therefore
\begin{equation}\label{m}
\mu_{ijk}\mu^{ijk}=\delta_{i}^{l}\delta_{j}^{m}-\delta_{i}^{m}\delta_{j}^{l},
\end{equation}
and we can rewrite $h_{ij}$ as
\begin{equation*}
h_{ij}=\frac{1}{8}R_{ilpq}\mu^{pqk}R_{kjrs}\mu^{rsl},
\end{equation*}
or equivalently as
\begin{equation*}
h_{ij}=\frac{1}{2}P^{rs}R_{irjs},
\end{equation*}
where $P^{ij}=\frac{1}{4}\mu^{irs}\mu^{jkl}R_{rskl}$ is the Einstein tensor. If we choose an orthonormal basis so that the eigenvalues of Einstein tensor are $a=R_{2323}$, $b=R_{1313}$, and $c=R_{1212}$, then the eigenvalues of $h_{ij}$ are $bc$, $ac$,  and $ab$ and the eigenvalues of $R_{ij}$ are $b+c$, $a+c$,  and $a+b$. Since  $R_{ijij}$ $(i\neq j)$ are sectional curvatures, so if manifold has positive sectional curvatures, then both $P_{ij}$ and $h_{ij}$ are positive definite.
\section{Short-time existence and uniqueness  for the ECCF}
In this section we investigate the short-time existence and uniqueness of the  solution to the ECCF.
\begin{theorem}\label{th1}
Let $(M,g_{0})$ be a compact 3-dimensional Riemannian manifold with positive sectional curvature.  Then there exists  a positive constant $T$ such that the evolution equation
\begin{equation}\label{pecc}
\frac{\partial g}{\partial t}=-2 h+2\rho Rg,\,\,\,\,\,\,g(0)=g_{0}.
\end{equation}
 has a unique solution metric $g(t)$ on $[0,T)$ for $\rho<\frac{P^{11}(g_{0})}{4} $.
\end{theorem}
\begin{proof}
We first compute the linearized operator $DL_{g_{0}}\tilde{g}_{ij}=\frac{d}{d\tau}|_{\tau=0}L(g_{0}+\tau \tilde{g})$ of the operator $L=-2h+2\rho Rg$ at a metric $g_{0}$ where  $\tilde{g}_{ij}$  denotes  a variation in the metric. The curvature tensor and the scalar curvature have the following linearizations \cite{P}
\begin{eqnarray*}
D(R_{ikjl})_{g_{0}}(\tilde{g})&=&\frac{1}{2}\left(\nabla_{i}\nabla_{l}\tilde{g}_{kj}-\nabla_{i}\nabla_{j}\tilde{g}_{kl}-\nabla_{k}\nabla_{l}\tilde{g}_{ij}
+\nabla_{k}\nabla_{j}\tilde{g}_{il}\right.\\ &&\left. -R_{ikl}^{\,\,\,\,p}\tilde{g}_{pj}-R_{ikj}^{\,\,\,\,p}\tilde{g}_{lp}\right)+\tilde{g}_{jm}R_{ikl}^{\,\,\,\,m},\\
DR_{g_{0}}(\tilde{g})&=&-\Delta(tr_{g_{0}}\tilde{g})+\nabla^{i}\nabla^{j}\tilde{g}_{ij}-R_{lp}\tilde{g}^{lp}.
\end{eqnarray*}
Fix a point $x\in M$, and consider normal coordinates at $x$. We can write the above operators at $x$,
\begin{eqnarray*}
D(R_{ikjl})_{g_{0}}(\tilde{g})&=&\frac{1}{2}\left( \frac{\partial^{2}\tilde{g}_{kj}}{\partial x^{i}\partial x^{l}}- \frac{\partial^{2}\tilde{g}_{kl}}{\partial x^{i}\partial x^{j}}+ \frac{\partial^{2}\tilde{g}_{il}}{\partial x^{k}\partial x^{j}}- \frac{\partial^{2}\tilde{g}_{ij}}{\partial x^{k}\partial x^{l}}\right)\\&&+\text{lower order terms},\\
DR_{g_{0}}(\tilde{g})&=&- \frac{\partial^{2}(tr_{g_{0}}\tilde{g})}{\partial x^{i}\partial x^{i}}+ \frac{\partial^{2}\tilde{g}_{ij}}{\partial x^{i}\partial x^{j}}+\text{lower order terms}.
\end{eqnarray*}
Hence, we get
\begin{equation*}
D(P^{ij})_{g_{0}}(\tilde{g})=\frac{1}{4}\mu^{ipq}\mu^{jrs}D(R_{pqrs})_{g_{0}}(\tilde{g})+\text{lower order terms}.
\end{equation*}
Therefore, using (\ref{m}) we obtain
\begin{eqnarray*}
D(h_{ij})_{g_{0}}(\tilde{g})&=&P^{kl}(g_{0})D(R_{ikjl})_{g_{0}}(\tilde{g})+\text{lower order terms}\\&=&
\frac{1}{2}P^{kl}(g_{0})\left( \frac{\partial^{2}\tilde{g}_{kj}}{\partial x^{i}\partial x^{l}}- \frac{\partial^{2}\tilde{g}_{kl}}{\partial x^{i}\partial x^{j}}+ \frac{\partial^{2}\tilde{g}_{il}}{\partial x^{k}\partial x^{j}}- \frac{\partial^{2}\tilde{g}_{ij}}{\partial x^{k}\partial x^{l}}\right)\\&&+\text{lower order terms},
\end{eqnarray*}
and
\begin{eqnarray*}
DL_{g_{0}}\tilde{g}_{ij}&=&-P^{kl}(g_{0})\left( \frac{\partial^{2}\tilde{g}_{kj}}{\partial x^{i}\partial x^{l}}- \frac{\partial^{2}\tilde{g}_{kl}}{\partial x^{i}\partial x^{j}}+ \frac{\partial^{2}\tilde{g}_{il}}{\partial x^{k}\partial x^{j}}- \frac{\partial^{2}\tilde{g}_{ij}}{\partial x^{k}\partial x^{l}}\right)\\&&+2\rho \left( - \frac{\partial^{2}(tr_{g_{0}}\tilde{g})}{\partial x^{k}\partial x^{k}}+ \frac{\partial^{2}\tilde{g}_{kl}}{\partial x^{k}\partial x^{l}}\right)(g_{0})_{ij}+\text{lower order terms}.
\end{eqnarray*}
Now, we obtain the symbol of the linear differential operator $DL_{g_{0}}\tilde{g}_{ij}$ in the direction of an arbitrary cotangent vector $\xi$ by replacing each derivative $\frac{\partial}{\partial x^{\alpha}}$ appearing in the higher order terms with $\xi_{\alpha}$:
\begin{eqnarray*}
\sigma(DL_{g_{0}})(\xi)\tilde{g}_{ij}&=&-P^{kl}(g_{0})\left( \xi_{i} \xi_{l}\tilde{g}_{kj}-  \xi_{i} \xi_{j}\tilde{g}_{kl}+  \xi_{k} \xi_{j}\tilde{g}_{il}-  \xi_{k} \xi_{l}\tilde{g}_{ij}\right)\\&&+2\rho \left( - \xi^{k} \xi_{k}(trv\tilde{g})+  \xi_{l} \xi_{k}\tilde{g}_{kl}\right)(g_{0})_{ij}.
\end{eqnarray*}
Since the principal symbol is homogeneous, we may assume that $\xi=(1,0,...,0)$ satisfies $\xi_{1}=1$ and $\xi_{i}=0$ for $i\neq1$. A simple computation shows that
\begin{eqnarray*}
\sigma(DL_{g_{0}})(\xi)\tilde{g}_{ij}&=&P^{11}(g_{0})\tilde{g}_{ij}+P^{lk}(g_{0})\left( \delta_{i1} \delta_{j1}\tilde{g}_{kl}-  \delta_{i1} \delta_{k1}\tilde{g}_{lj}-  \delta_{l1} \delta_{j1}\tilde{g}_{ik}\right)\\&&+2\rho \left( - \delta^{ij} (tr_{g_{0}}\tilde{g})+  \delta_{ij} \tilde{g}_{11}\right).
\end{eqnarray*}
In coordinate system we have
\begin{equation*}
(\tilde{g}_{11},\tilde{g}_{12},\tilde{g}_{13},\tilde{g}_{22},\tilde{g}_{33},\tilde{g}_{23})
\end{equation*}
we deduce  that
\begin{equation*}
\sigma(DL_{g_{0}})(\xi)=\left(
  \begin{array}{cccccc}
    0 & 0 & 0& P^{22}(g_{0})-2\rho &  P^{33}(g_{0})-2\rho & 2P^{23}(g_{0}) \\
    0 & 0 & 0 & -P^{12}(g_{0}) & 0 & -P^{13}(g_{0}) \\
    0 & 0 & 0 & 0 & -P^{13}(g_{0}) & -P^{12}(g_{0}) \\
    0 & 0 & 0 &  P^{11}(g_{0})-2\rho & -2\rho & 0 \\
    0 & 0 & 0 & -2\rho &  P^{11}(g_{0})-2\rho & 0 \\
    0 & 0 & 0& 0 & 0 & P^{11}(g_{0}) \\
  \end{array}
\right).
\end{equation*}
The eigenvalues of this matrix are $0$ with multiplicity 3, $P^{11}(g_{0})$ with multiplicity 2 and $P^{11}(g_{0})-4\rho$ with multiplicity $1$.  The sectional curvature  is  assumed to be positive, so the system (\ref{pecc}) is weakly parabolic. In order to eradicate the zero eigenvalues, we apply the so-called DeTurk's trick \cite{D} to show that the flow (\ref{pecc}) is equivalent to a strictly parabolic initial-value problem, under the action of diffeomorphisms group  of $M$. Let $V$ be DeTurk's vector field defined by
\begin{equation*}
V^{i}(g)=-g_{0}^{ij}g^{kl}\nabla_{k}\left( \frac{1}{2}tr_{g}(g_{0})g_{jl}-(g_{0})_{jl}\right),
\end{equation*}
then
\begin{eqnarray*}
(\mathcal{L}_{V}g)_{ij}&=&\nabla_{j}V_{i}+\nabla_{i}V_{j}\\&=&g^{kl}\left(
 \frac{\partial^{2}{g}_{ki}}{\partial x^{l}\partial x^{j}}- \frac{\partial^{2}{g}_{kl}}{\partial x^{i}\partial x^{j}}+ \frac{\partial^{2}{g}_{kj}}{\partial x^{l}\partial x^{i}}
 \right)+\text{lower order terms}.
\end{eqnarray*}
 With the same notation used above, the principal symbol of operator $\mathcal{L}_{V}$ is given by
\begin{equation*}
\sigma(D\mathcal{L}_{V})_{g_{0}}(\xi)\tilde{g}_{ij}=\delta_{i1}\delta_{j1}tr_{g_{0}}(\tilde{g})-\delta_{i1}\tilde{g}_{j1}-\delta_{j1}\tilde{g}_{i1}.
\end{equation*}
Now we consider  the initial-value problem
\begin{equation}\label{dpecc}
\frac{\partial g}{\partial t}=-2 h+2\rho Rg+\mathcal{L}_{V}g,\,\,\,\,\,\,g(0)=g_{0}.
\end{equation}
The computations above show that the linearized of operator $L-\mathcal{L}_{V}$ has principal symbol in the direction $\xi$ is given by
\begin{equation*}
\sigma(DL_{g_{0}})(\xi)=\left(
  \begin{array}{cccccc}
    1 & 0 & 0& P^{22}(g_{0})-2\rho-1 &  P^{33}(g_{0})-2\rho-1 & 2P^{23}(g_{0}) \\
    0 & 1 & 0 & -P^{12}(g_{0}) & 0 & -P^{13}(g_{0}) \\
    0 & 0 & 1 & 0 & -P^{13}(g_{0}) & -P^{12}(g_{0}) \\
    0 & 0 & 0 &  P^{11}(g_{0})-2\rho & -2\rho & 0 \\
    0 & 0 & 0 & -2\rho &  P^{11}(g_{0})-2\rho & 0 \\
    0 & 0 & 0& 0 & 0 & P^{11}(g_{0}) \\
  \end{array}
\right).
\end{equation*}
This matrix has $3$ eigenvalues equal  to $1$, $2$ eigenvalues equal to $P^{11}(g_{0})$, and one eigenvalue equal to $P^{11}(g_{0})-4\rho$. Therefore the flow (\ref{dpecc}) is a strictly parabolic system and a sufficient condition for the short-time existence of  a solution (\ref{dpecc}) is that $\rho<\frac{P^{11}(g_{0})}{4} $, hence a unique solution exists  for a shot-time  for    (\ref{dpecc})  by standard parabolic theory.\\
In what  follows,  we will show how to go from a solution of the (\ref{dpecc}) back to a solution for (\ref{pecc}). Define a one-parameter family of maps $\phi_{t}:M\to M$ by
\begin{equation*}
\frac{\partial}{\partial t}\phi_{t}(x)=-V(\phi_{t}(x),t),\,\,\,\phi_{0}=Id_{M}.
\end{equation*}
 From \cite{CK} the maps $\phi_{t}$ exist  and are  diffeomorphisms as long as the solution $g(t)$ exists. The same way as the Ricci flow \cite{D, P, CK},  if $g(t)$ be a solution to (\ref{dpecc}) then $\hat{g}(t)=\phi_{t}^{*}g(t)$ is a solution to (\ref{pecc}). Now, assume  we have a  solution $\bar{g}(t)$ to (\ref{pecc}). Let $\psi_{t}$ be the solution to the harmonic map flow
\begin{equation*}
\frac{\partial\psi_{t}}{\partial t}=\Delta_{\bar{g}(t),\bar{\bar{g}}}\psi_{t},\,\,\,\psi_{0}=Id_{M},
\end{equation*}
where $\bar{\bar{g}}$ is any reference metric and by \cite{E} the maps $\psi_{t}$ exist and are diffeomorphisms. By direct computation, we see that $g(t)=(\psi_{t})_{*}\bar{g}(t)$  is a solution for (\ref{dpecc}).  For uniqueness, if we have two solutions $\bar{g}_{1}(t)$ and $\bar{g}_{2}(t)$ of (\ref{pecc}) with the same initial data, then  using above method, we get two solutions ${g}_{1}(t)$ and ${g}_{2}(t)$ of (\ref{dpecc}) with the same initial data. By uniqueness of the  solution to (\ref{dpecc}), we conclude that $g_{1}(t)=g_{2}(t)$ and hence   $\bar{g}_{1}(t)=\bar{g}_{2}(t)$.
\end{proof}
\begin{remark}
Let $(M,g_{0})$ be a compact 3-dimensional Riemannian manifold with negative sectional curvature.  Similarly to  the proof of theorem \ref{th1}, we can show that  there exists  a positive constant $T$ such that the evolution equation
\begin{equation}\label{necc}
\frac{\partial g}{\partial t}=2 h+2\rho Rg,\,\,\,\,\,\,g(0)=g_{0}.
\end{equation}
 has a unique solution metric $g(t)$ on $[0,T)$ for $\rho<-\frac{P^{11}(g_{0})}{2} $.
\end{remark}
\section{Example}
\begin{example}
Let $(M, g(0))$ be a  closed $3$-dimensional  Riemannian manifold, the initial metric $g(0)$  is Einstein metric that is for some constant $\lambda$ it satisfies
\begin{equation} R_{ij}(0)=\lambda g_{ij}(0). \end{equation}
Since, the initial metric is Einstein for some constant $\lambda$, let   $g_{ij}(t,x)=c(t) g_{ij}(0)$. By the definition of the Ricci tensor we obtain
\begin{eqnarray*}
&&R_{ij}(t)=R_{ij}(0)=\lambda g_{ij}(0),\\
&&R(t)=\frac{3\lambda}{c(t)}.\\
 \end{eqnarray*}
Hence,
\begin{eqnarray*}
&&P^{ij}(t)=\frac{2c^{2}(t)-3\lambda^{2}}{2\lambda c^{2}(t)}g^{ij}(0),\\
&&V_{ij}(t)=\frac{2\lambda c^{2}(t)}{2c^{2}(t)-3\lambda^{2}}g_{ij}(0),\\
&&h_{ij}=(\frac{2c^{2}(t)-3\lambda^{2}}{2\lambda c^{2}(t)})^{2}c^{-3}(t) g_{ij}(0),\\
&&-2h_{ij}+2\rho R g_{ij}=\left( -2\left(\frac{2c^{2}(t)-3\lambda^{2}}{2\lambda c^{2}(t)}\right)^{2}c^{-3}(t)+6\rho a\right)g_{ij}(0).
 \end{eqnarray*}
In the present situation, the equation (\ref{pecc}) becomes
\begin{equation}
\frac{\partial(c(t)g_{ij}(0))}{\partial t}=\left( -2\left(\frac{2c^{2}(t)-3\lambda^{2}}{2\lambda c^{2}(t)}\right)^{2}c^{-3}(t)+6\rho a\right)g_{ij}(0),
 \end{equation}
this gives an ODE of first  order as follows
\begin{equation}
\frac{d c(t)}{d t}=\left( -2\left(\frac{2c^{2}(t)-3\lambda^{2}}{2\lambda c^{2}(t)}\right)^{2}c^{-3}(t)+6\rho a\right),\,\,\,\,c(0)=1.
 \end{equation}

Therefore the solution of    (\ref{pecc}) flow remains Einstein.
\end{example}


\end{document}